\documentclass[11 pt]{amsart}
\usepackage{amscd,amsfonts,amssymb,amsmath}
\usepackage{hyperref}

\usepackage{tikz}
\usepackage[margin=3.7 cm]{geometry}

\newtheorem{theorem}{Theorem}[section]

\newtheorem{prop}[theorem]{Proposition}

\theoremstyle{definition}
\newtheorem{definition}[theorem]{Definition}
\newtheorem{example}[theorem]{Example}

\numberwithin{equation}{subsection}
\theoremstyle{plain}

\newtheorem{question}{Question}
\newtheorem{problem}{Problem}
\newtheorem{corollary}[theorem]{Corollary}

\numberwithin{equation}{section}

\begin{document}
\title{Compatible actions and  non-abelian tensor products}
\author[V. G. Bardakov]{Valeriy G. Bardakov}
\author[M. V. Neshchadim]{Mikhail V. Neshchadim}

\date{\today}
\address{Sobolev Institute of Mathematics, Novosibirsk 630090, Russia,}
\address{Novosibirsk State University, Novosibirsk 630090, Russia,}
\address{Novosibirsk State Agrarian University, Dobrolyubova street, 160,  Novosibirsk, 630039, Russia,}
\email{bardakov@math.nsc.ru}

\address{Sobolev Institute of Mathematics and Novosibirsk State University, Novosibirsk 630090, Russia,}
\email{neshch@math.nsc.ru}



\begin{abstract}
For a pair of groups $G, H$ we study pairs of  actions $G$ on $H$ and $H$ on $G$  such that these pairs  are compatible and  non-abelian tensor products $G \otimes H$ are defined.
\end{abstract}
\maketitle

\section{Introduction}

R. Brown and J.-L. Loday \cite{BL, BL1}
introduced the non-abelian tensor product $G \otimes H$ for a pair of groups $G$ and $H$
 following  works of C.~Miller \cite{Mil},  and A.~S.-T.~Lue \cite{Lue}.
The investigation of the non-abelian tensor product from a group theoretical
point of view started with a paper by R.~Brown, D.~L.~Johnson, and E.~F.~Robertson \cite{BJR}.

The non-abelian tensor product $G \otimes H$ depends not only on the groups $G$ and $H$ but also on the action of $G$ on $H$ and on the action of $H$ on $G$. Moreover these actions must be compatible  (see the definition in  Section 2). In the present paper we study the following question: what  actions are compatible?

The paper is organized as follows. In Section 2, we recall a definition of non-abelian tensor product, formulate some its properties and give an answer on a question of V.~Thomas, proving that there are nilpotent group $G$ and some group $H$ such that in $G \otimes H$ the derivative subgroup $[G, H]$ is equal to $G$. In the Section 3 we study the following question: Let a group $H$ acts on a group $G$ by automorphisms, is it possible to define an action of $G$ on $H$ such that this pair  of actions are compatible?
Some necessary conditions for compatibility of actions will be given and in some cases will be prove a formula for the second action if the first one is given.
In the Section 4 we
construct pairs compatible actions for arbitrary groups and for 2-step nilpotent groups give a particular answer on the question from Section 3.
In Section 5 we study groups of the form $G \otimes \mathbb{Z}_2$ and describe compatible actions.

\section{Preliminaries}

In this article we will use the following notations.  For elements $x$, $y$ in a group
$G$, the conjugation of $x$ by $y$ is $x^y = y^{-1} x y$; and the commutator of $x$ and $y$ is
$[x, y] = x^{-1} x^y = x^{-1} y^{-1} x y$. We write $G'$ for the derived subgroup of $G$, i.e. $G' = [G, G]$; $G^{ab}$ for the
abelianized group $G / G'$; the second hypercenter $\zeta_2 G$ of $G$ is the  subgroup of $G$ such that
$$
\zeta_2 G  / \zeta_1 G = \zeta_1 (G / \zeta_1 G),
$$
where $\zeta_1 G = Z(G)$ is the center of a group $G$.

Recall the definition of the non-abelian tensor product $G \otimes H$ of groups $G$ and $H$ (see \cite{BL, BL1}).
It is defined for a pair of groups $G$ and $H$ where each one acts on the other (on
right)
$$
G \times H \longrightarrow G, ~~(g, h) \mapsto g^h; ~~~H \times G \longrightarrow H, ~~(h, g) \mapsto h^g
$$
and on itself by conjugation, in such a way that for all $g, g_1 \in G$ and $h, h_1 \in H,$
$$
g^{(h^{g_1})} = \left( \left( g^{g_1^{-1}} \right)^h \right)^{g_1}~~
\mbox{and}~~
h^{(g^{h_1})} = \left( \left( h^{h_1^{-1}} \right)^g \right)^{h_1}.
$$
In this situation we say that $G$ and $H$ act {\it compatibly} on each other. The {\it non-abelian
tensor product} $G \otimes H$ is the group generated by all symbols $g \otimes h$, $g \in G$, $h \in H$,
subject to the relations
$$
g g_1 \otimes h = (g^{g_1} \otimes h^{g_1}) (g_1 \otimes h)~ \mbox{and}~~ g \otimes h h_1 = (g \otimes h_1) (g^{h_1} \otimes h^{h_1})
$$
for all $g, g_1 \in G$, $h, h_1 \in H$.

In particular, as the conjugation action of a group $G$ on itself is compatible, then
the tensor square $G \otimes G$ of a group $G$ may always be defined. Also, the tensor product $G \otimes H$ is defined if $G$ and $H$ are two normal subgroups of some group $M$ and actions are conjugations in $M$.

\medskip

The following proposition is well known. We give a proof only for fullness.

\begin{prop} \label{P2.2}
1) Let  $G$ and $H$ be abelian groups. Independently on the action of $G$ on $H$ and $H$ on $G$, the group $G \otimes H$ is abelian.

2) (See \cite[Proposition 2.4]{BL1}) Let  $G$ and $H$ be arbitrary groups. If the actions of $G$ on $H$ and $H$ on $G$ are trivial, then the group $G \otimes H \cong G^{ab} \otimes_{\mathbb{Z}} H^{ab}$ is the abelian tensor product.
\end{prop}

\begin{proof}
1) We have the equality
$$
(g \otimes h)^{g_1 \otimes h_1} = g^{[g_1, h_1]} \otimes h^{[g_1, h_1]},
$$
where $g^{[g_1, h_1]}$ is the action of the commutator $[g_1, h_1] \in G$ by conjugation on $g$, but $G$ is abelian and $g^{[g_1, h_1]} = g$. Analogously,
$h^{[g_1, h_1]} = h$. Hence, $G \otimes H$ is abelian.

2) From the previous formula and triviality  actions we have
$$
g^{[g_1, h_1]} = g^{g_1^{-1} h_1^{-1} g_1 h_1} = \left( g^{g_1^{-1}} \right)^{h_1^{-1} g_1 h_1} = \left( g^{g_1^{-1}} \right)^{ g_1 h_1} =
\left( g^{g_1^{-1}} \right)^{ g_1 h_1} =  g^{h_1} = g.
$$
Analogously,
$h^{[g_1, h_1]} = h$. Hence, $G \otimes H$ is abelian.
\end{proof}

\medskip

Remind presentation of non-abelian tensor product as a central extension (see \cite{DLT}).  The {\it derivative subgroup} of $G$ by $H$ is called the following subgroup
$$
D_H(G) = [G, H] = \langle g^{-1} g^h ~|~g\in G, h \in H \rangle.
$$
The map $\kappa : G \otimes H \longrightarrow D_H(G)$ defined by $\kappa (g \otimes h) = g^{-1} g^h$ is a homomorphism, its kernel $A = \ker(\kappa)$ is the central subgroup of
$G \otimes H$  and $G$ acts on $G \otimes H$ by the rule $(g \otimes h)^x = g^x \otimes h^x$, $x \in G$, i.e. there exists the short exact sequence
$$
1 \longrightarrow A \longrightarrow G \otimes H  \longrightarrow D_H (G) \longrightarrow 1.
$$
In this case $A$ can be viewed as $\mathbb{Z} [D_H (G)]$-module via conjugation in $G \otimes H$, i.e. under the action induced by setting
$$
a \cdot g = x^{-1} a x,~~a \in A, x \in G \otimes H, \kappa(x) = g.
$$

\medskip

The following proposition gives an answer on the following question: is there non-abelian tensor product $G \otimes H$ such that $[G, H] = G$? which of V.~Thomas formulated in some letter to the authors.

\begin{prop}
Let   $G = F_n / \gamma_k F_n$, $k \geq 2$, be a free nilpotent group of rank  $n \geq 2$ and  $H = \mathrm{Aut} (G)$ is its automorphism group. Then   $D_H(G) = [G, H] = G$.
\end{prop}

\begin{proof}

Let  $F_n$ be a free group of rank $n\geq 2$ with the basis $x_1, \ldots, x_n$,
$G = F_n / \gamma_k F_n$ be a free $k-1$-step nilpotent group for $k\geq 2$.
 Let $G$ acts trivially on  $H$ and
elements of $H$ act by automorphisms on $G$. It is easy to see that these actions are compatible.

Let us show that in this case $[G, H]=G$. To do it, let us prove that $x_1$ lies in $[G, H]$.
Take  $\varphi_1 \in H=\mathrm{Aut} (G)$, which  acts on the generators of $G$
by the rules:
$$
x_1^{\varphi_1} = x_1, ~~x_2^{\varphi_1} = x_2 x_1, ~~x_3^{\varphi_1} = x_3, ~~\ldots,~~ x_n^{\varphi_1} = x_n.
$$
Then
$$
x_1^{-1} x_1^{\varphi_1} = 1, ~~x_2^{-1} x_2^{\varphi_1} = x_1,~~x_3^{-1} x_3^{\varphi_1} = 1, ~~ \ldots,~~ x_n^{-1}x_n^{\varphi_1} = 1.
$$
Hence the generator  $x_1$ lies in $[G, H]$.
Analogously, $x_2, x_3, \ldots, x_n$ lie  in $[G,H]$. This completes the proof.
\end{proof}

\section{What actions are compatible?}

In this section we study

\begin{question} \label{q1}
Let a group $H$ acts on a group $G$ by automorphisms. Is it possible to define an action of $G$ on $H$ such that this pair  of actions are compatible?
\end{question}

Consider some examples.

\begin{example}
Let us take $G = \{ 1, a, a^2 \} \cong \mathbb{Z}_3$, $H = \{ 1, b, b^2 \} \cong \mathbb{Z}_3$. In dependence on actions we have three cases.

1) If the action of $H$ on $G$ and the action of $G$ on $H$ are trivial, then by the second part of Proposition \ref{P2.2} $G \otimes H = \mathbb{Z}_3 \otimes_{\mathbb{Z}} \mathbb{Z}_3 \cong \mathbb{Z}_3$ is abelian tensor product.

2) Let $H$ acts non-trivially on $G$, i.e. $a^b = a^2$ and the action $G$ on $H$ is trivial. It is not difficult to check that $G$ and $H$ act compatibly on each other. To find $D_H(G) = [G, H]$  we calculate
$$
[a, b] = a^{-1} a^b = a^2 a^2 = a.
$$
Hence, $D_H(G) = G$. But $D_G(H) = 1$.

By the definition, $G \otimes H$ is generated by elements
$$
a \otimes b, a^2 \otimes b, a \otimes b^2, a^2 \otimes b^2.
$$
Using the defining relations:
$$
g g_1 \otimes h = (g^{g_1} \otimes h^{g_1}) (g_1 \otimes h),~~~g \otimes h h_1 = (g \otimes h_1) (g^{h_1} \otimes h^{h_1}),
$$
we find
$$
a^2 \otimes b = (a^{a} \otimes b^{a}) (a \otimes b) = (a \otimes b)^2,~~~
a \otimes b^2 = (a \otimes b) (a^b \otimes b^b) = (a \otimes b) (a^2 \otimes b) = (a \otimes b)^3.
$$
On the other side
$$
1 = a^2 a \otimes b = (a^{2} \otimes b^{a}) (a \otimes b) = (a \otimes b)^3.
$$
Hence,
$$
a \otimes b^2 =  a^2 \otimes b^2 = 1
$$
and in this case we have the same result: $\mathbb{Z}_3 \otimes \mathbb{Z}_3 = \mathbb{Z}_3$.

3) Let $H$ acts non-trivially on $G$, i.e. $a^b = a^2$ and  $G$ acts non-trivially on $H$. In this case $G$ and $H$ act non-compatibly on each other. Indeed,
$$
a^{(b^a)} = a^{b^2} = (a^2)^b = a,
$$
but
$$
\left( \left( a^{a^{-1}} \right) ^b \right)^a = (a^{b})^a = (a^2)^2 = a^2.
$$
Hence, the  equality
$$
a^{(b^a)} = \left( \left( a^{a^{-1}} \right) ^b \right)^a
$$
does not hold.
\end{example}

Let $G$, $H$ be some groups. Actions of
$G$ on $H$ and  $H$ on  $G$ are defined by homomorphisms
$$
\beta : G \rightarrow \mathrm{Aut} (H),\quad
 \alpha : H \rightarrow \mathrm{Aut} (G),
$$
and by definition
$$
g^h=g^{\alpha(h)}, \quad h^g=h^{\beta(g)},\quad g\in G, h\in H.
$$
The actions $(\alpha,\beta)$ are compatible, if
$$
g^{\alpha\left(h^{\beta(g_1)}\right)}= \left( \left( g^{g_1^{-1}} \right)^{\alpha(h)} \right)^{g_1}
$$
and
$$
h^{\beta\left(g^{\alpha(h_1)}\right)}= \left( \left( h^{h_1^{-1}} \right)^{\beta(g)} \right)^{h_1}
$$
for all $g, g_1 \in G$, $h, h_1 \in H$. In this case we will say that the pair $(\alpha,\beta)$ {\it is compatible}.

Rewrite these equalities in the form
$$
\alpha\left(h^{\beta(g_1)}\right)= \widehat{g_1}^{-1}\alpha(h)\widehat{g_1}
\eqno{(1)}
$$
and
$$
\beta\left(g^{\alpha(h_1)}\right)= \widehat{h_1}^{-1}\beta(g)\widehat{h_1},
\eqno{(2)}
$$
where $\widehat{g}$ is the inner automorphism of  $G$ which is induced by conjugation of $g$, i.e.
$$
\widehat{g} : g_1 \mapsto g^{-1}g_1g, ~~g, g_1 \in G,
$$
 and analogously,
$\widehat{h}$ is the inner automorphism of $H$  which is induced by the conjugation of  $h$, i.e.
$$
\widehat{h} : h_1 \mapsto h^{-1}h_1h, ~~h, h_1 \in H.
$$

\medskip

\begin{theorem}\label{t1}
1) If the pair $(\alpha,\beta)$ defines compatible actions of $H$ on $G$ and $G$ on $H$, then the following inclusions hold
$$
N_{\mathrm{Aut} (G)}(\alpha(H))\geq \mathrm{Inn} (G),\quad
 N_{\mathrm{Aut} (H)}(\beta(G))\geq \mathrm{Inn} (H).
$$
Here $\mathrm{Inn} (G)$ and $\mathrm{Inn} (H)$ are the subgroups of inner automorphisms.

2) If $ \alpha : H \rightarrow \mathrm{Aut} (G)$ is an embedding and
$N_{\mathrm{Aut} (G)}(\alpha(H))\geq \mathrm{Inn} (G)$, then
defining $\beta : G \rightarrow \mathrm{Aut} (H)$ by the formula
$$
\beta(g): h \mapsto  \alpha^{-1}\left(\widehat{g}^{-1}\alpha(h)\widehat{g}\right),\quad h \in H,
$$
we get the compatible actions $(\alpha,\beta)$.
\end{theorem}

\begin{proof}
The first claim immediately  follows from the relations  (1), (2).

To prove the second claim it is enough to check (2), or that is equivalent, the equality
$$
h^{\beta\left(g^{\alpha(h_1)}\right)}= \left( \left(h^{h_1^{-1}} \right)^{\beta(g)} \right)^{h_1}.
\eqno{(3)}
$$
Using the definition $\beta$, rewrite the left side of
 (3):
$$
h^{\beta\left(g^{\alpha(h_1)}\right)}=
 \alpha^{-1}\left( \widehat{g^{\alpha(h_1)}}^{-1}  \alpha(h) \widehat{g^{\alpha(h_1)}} \right).
\eqno{(4)}
$$
 Rewrite the right  side of (3):
$$
\left( \left( h^{h_1^{-1}} \right)^{\beta(g)} \right)^{h_1}=
  h_1^{-1}(h_1 h h_1^{-1})^{\beta(g)}h_1=
    h_1^{-1} \alpha^{-1}( \widehat{g}^{-1}  \alpha(h_1 h h_1^{-1}) \widehat{g} )  h_1.
    \eqno{(5)}
$$
From (4) and (5):
$$
 \alpha^{-1}\left( \widehat{g^{\alpha(h_1)}}^{-1}  \alpha(h) \widehat{g^{\alpha(h_1)}} \right)
=
    h_1^{-1} \alpha^{-1}( \widehat{g}^{-1}  \alpha(h_1 h h_1^{-1}) \widehat{g} )  h_1.
$$
Using the homomorphism $\alpha$:
$$
 \widehat{g^{\alpha(h_1)}}^{-1}  \alpha(h) \widehat{g^{\alpha(h_1)}}
=
   \alpha \left( h_1^{-1}\alpha^{-1}( \widehat{g}^{-1}  \alpha(h_1 h h_1^{-1}) \widehat{g} )  h_1 \right)=
$$
$$
=  \alpha(h_1)^{-1}\widehat{g}^{-1}  \alpha(h_1 h h_1^{-1}) \widehat{g} \alpha(h_1)=
$$
$$
=  \alpha(h_1)^{-1}\widehat{g}^{-1}  \alpha(h_1) \alpha(h) \alpha(h_1)^{-1}) \widehat{g} \alpha(h_1)=
 \widehat{g^{\alpha(h_1)}}^{-1}  \alpha(h) \widehat{g^{\alpha(h_1)}}.
$$
In the last equality we used the formula
$$
\alpha(h_1)^{-1}\widehat{g}  \alpha(h_1) = \widehat{g^{\alpha(h_1)}}.
$$
Hence, the equality (3) holds.
\end{proof}

\begin{question}
Are the inclusions
$$
N_{\mathrm{Aut} (G)}(\alpha(H))\geq \mathrm{Inn} (G),\quad
 N_{\mathrm{Aut} (H)}(\beta(G))\geq \mathrm{Inn} (H)
$$
sufficient  for compatibility of the pare $(\alpha, \beta)$?
\end{question}

\section{Compatible actions for nilpotent groups}

At first, recall the following definition.

\begin{definition}
Let $G$ and  $H$ be groups and $G_1 \unlhd G$,   $H_1 \unlhd H$ are their normal subgroups.
We will say that $G$ is {\it comparable with $H$ with respect to the pare}   $(G_1, H_1)$,
if there are homomorphisms
$$
\varphi : G \longrightarrow H,\quad \psi : H \longrightarrow G,
$$
such that что
$$
x\equiv \psi \varphi(x) (\mathrm{mod}\, G_1), \quad
y\equiv \varphi \psi(y) (\mathrm{mod}\, H_1)
$$
for all $x\in G$, $y \in H$, i.e.
$$
 x^{-1} \cdot \psi \varphi(x) \in G_1, \quad y^{-1} \cdot \varphi \psi(y) \in H_1.
$$

\end{definition}

Note that if $G_1=1$, $H_1=1$, then  $\varphi$, $\psi$
are mutually inverse isomorphisms.

The following theorem holds.

\begin{theorem}
Let $G$, $H$ be groups and there exist homomorphisms
$$
\varphi : G \longrightarrow H,\quad \psi : H \longrightarrow G,
$$
such that
$$
x\equiv \psi\varphi(x) (\mathrm{mod}\, \zeta_2 G), \quad
y\equiv \varphi\psi(y) (\mathrm{mod}\, \zeta_2 H)
$$
for all $x\in G$, $y \in H$.
Then the action of $G$ on $H$ and the action of  $H$ on $G$ by the rules
$$
x^y=\psi(y)^{-1}x\psi(y), \quad y^x= \varphi(x)^{-1}y\varphi(x), \quad x\in G, ~y \in H,
$$
are compatible, i.e. the following equalities hold
$$
x^{(y^{x_1})}=((x^{x_1^{-1}})^y)^{x_1},\quad
  y^{(x^{y_1})}=((y^{y_1^{-1}})^x)^{y_1},\quad x, x_1\in G, ~y, y_1 \in H.
$$
\end{theorem}

\begin{proof}
Let us prove that  the following relation holds
$$
x^{(y^{x_1})}=((x^{x_1^{-1}})^y)^{x_1}.
$$
For this denote the left hand side of this relation by  $L$ and transform it:
$$
L=x^{(y^{x_1})}=x^{\varphi(x_1)^{-1}y\varphi(x_1)}=
 \psi(\varphi(x_1)^{-1}y^{-1}\varphi(x_1))x\psi(\varphi(x_1)^{-1}y\varphi(x_1))=
$$
$$
 =(\psi\varphi(x_1))^{-1}\psi(y)^{-1}(\psi\varphi(x_1))x (\psi\varphi(x_1))^{-1}\psi(y)(\psi\varphi(x_1))=
$$
$$
 =(c(x_1)^{-1} x_1^{-1} \psi(y)^{-1}x_1 c(x_1)) x (c(x_1)^{-1} x_1^{-1} \psi(y)x_1 c(x_1)).
$$
Here $\psi\varphi(x_1)=x_1 c(x_1)$, $c(x_1)\in \zeta_2\,G$.
Since  $c(x_1)\in \zeta_2\,G$, then the commutator  $[x_1^{-1} \psi(y) x_1, c(x_1)]$
lies in the center of  $G$. Hence
$$
L=x^{x_1^{-1} \psi(y) x_1}.
$$
Denote the right hand side of this relation by  $R$ and transform it:
$$
R=((x^{x_1^{-1}})^y)^{x_1}=((x^{x_1^{-1}})^\psi(y))^{x_1}=x^{x_1^{-1} \psi(y) x_1}.
$$
We see that  $L=R$, i.e. the first relation from the definition of compatible action holds. The checking of the second relation is the similar.
\end{proof}

From this theorem we have particular  answer on Question \ref{q1} for 2-step nilpotent groups.

\begin{corollary} \label{Cor1}
If $G$, $H$ are 2-step nilpotent groups, then any pare of homomorphisms
$$
\varphi : G \longrightarrow H,\quad \psi : H \longrightarrow G
$$
define the compatible action.
\end{corollary}

\begin{problem} Let $G$ and $H$ be free 2-step nilpotent groups. By Corollary \ref{Cor1}, any pair of homomorphisms
 $(\varphi, \psi)$,
where $\varphi \in \mbox{Hom} (G,H)$, $\psi\in \mbox{Hom} (H,G)$ defines a tensor product $M(\varphi, \psi)=G \otimes H$. Give a classification of the groups $M(\varphi, \psi)$.
\end{problem}

Note that for arbitrary groups Corollary \ref{Cor1} does not hold. Indeed,
 let  $G=\langle x_1, x_2\rangle$,
$H=\langle y_1, y_2\rangle$ be free groups of rank 2. Define the homomorphisms
$$
\varphi : G \longrightarrow H,\quad \psi : H \longrightarrow G
$$
 by the rules
$$
\varphi(x_1)=y_1, \,\, \varphi(x_2)=y_2,\quad \psi(y_1)=\psi(y_2)=1.
$$
Then
$$
y_2^{x_1}=y_2^{\varphi(x_1)}=y_2^{y_1}\neq y_2,
$$
i.e. the conditions of compatible actions does not hold.

\section{Tensor products $G \otimes \mathbb{Z}_2$ }

Note that the group  $\mathrm{Aut}(\mathbb{Z}_2)$ is trivial and hence, any group
 $G$ acts on  $\mathbb{Z}_2$ only trivially.

This section is devoted to the answer on the following question.

\begin{question}
Let $G$ be a group and $\psi \in \mathrm{Aut}(G)$ be an automorphism of order 2. Let $\mathbb{Z}_2 = \langle \varphi \rangle$ and $\alpha : \mathbb{Z}_2 \longrightarrow \mathrm{Aut}(G)$ such that $\alpha(\varphi) = \psi$. Under what conditions the pare $(\alpha, 1)$ is compatible?
\end{question}

 If $\psi \in \mathrm{Aut}(G)$ is trivial automorphism, then by the second part of Proposition~\ref{P2.2} $G \otimes \mathbb{Z}_2 = G^{ab} \otimes_{\mathbb{Z}} \mathbb{Z}_2$ is an abelian tensor product. In the general case we have

\begin{prop}
Let

1) $G$ be a group,

2) $\mathbb{Z}_2= \langle \varphi \rangle$ be a cyclic group of order two with the generator $\varphi$,

3) $\alpha : \mathbb{Z}_2  \longrightarrow \mathrm{Aut} (G)$ be a homomorphism, $\beta=1 : G \rightarrow \mathrm{Aut} (\mathbb{Z}_2)$ be the trivial homomorphism,

Then the pare  of actions $(\alpha, \beta)$ is compatible
if and only if for any  $g\in G$ holds
$$
g^{\alpha(\varphi)}=g c(g),
$$
where $c(g)$ is a central element of $G$ such that
$c(g)^{\alpha(\varphi)}=c(g)^{-1}$. In particular, if the center of $G$ is trivial, then $G \otimes \mathbb{Z}_2 = G^{ab} \otimes_{\mathbb{Z}} \mathbb{Z}_2$.
\end{prop}

\begin{proof}
Since $\mathrm{Inn} (G)$ normalizes  $\alpha(\mathbb{Z}_2)$, then for every
 $g\in G$ holds
$$
\widehat{g}^{-1} \alpha(\varphi) \widehat{g}= \alpha(\varphi).
$$
Using this equality for arbitrary element $x \in G$ we get
$$
g^{-1} g^{\alpha(\varphi)} x^{\alpha(\varphi)}(g^{-1} g^{\alpha(\varphi)})^{-1}= x^{\alpha(\varphi)}.
$$
Since $x^{\alpha(\varphi)}$ is an arbitrary element of  $G$,
then $c(g)$ is a central element of $G$. Applying  $\alpha(\varphi)$
to the equality $g^{\alpha(\varphi)}=gc(g)$ we have
$$
g=g^{\alpha(\varphi)^2}=g^{\alpha(\varphi)}c(g)^{\alpha(\varphi)}=g c(g)c(g)^{\alpha(\varphi)},
$$
that is $c(g)^{\alpha(\varphi)}=c(g)^{-1}$.
\end{proof}

For an arbitrary abelian group $A$ we know that $A \otimes_{\mathbb{Z}} \mathbb{Z} = A$. The following proposition is some analog of this property for non-abelian tensor product.

\begin{prop}
Let $A$ be an abelian group, $\mathbb{Z}_2 = \langle \varphi \rangle$ is the cyclic group of order $2$
and $\varphi$ acts on the elements of $A$ by the following manner
$$
a^{\varphi} = a^{-1}, \quad a\in A.
$$
Then the non-abelian tensor product $A \otimes \mathbb{Z}_2$ is defined and there is an isomorphism
$$
A \otimes \mathbb{Z}_2 \cong A.
$$
\end{prop}

\begin{proof}
It is not difficult to check that defined actions are compatible.

Since $A$ acts on $\mathbb{Z}_2$ trivially and  $A$ is abelian, then the defining relations of the tensor product:
$$
a a_1 \otimes h = (a^{a_1} \otimes h^{a_1}) (a_1 \otimes h),\quad a, a_1 \in A, \quad h \in \mathbb{Z}_2,
$$
have the form
$$
a a_1 \otimes h = (a \otimes h) (a_1 \otimes h)=(a_1 \otimes h) (a \otimes h).
\eqno{(1)}
$$
The relations
$$
a \otimes h h_1 = (a \otimes h_1) (a^{h_1} \otimes h^{h_1}),\quad a \in A, \quad h, h_1 \in \mathbb{Z}_2,
$$
give only one non-trivial relation
$$
1=a \otimes \varphi^2 = (a \otimes \varphi) (a^{-1} \otimes \varphi),\quad a \in A,
$$
which follows from  (1).

Since the set of relations (1) is a full system of relations for
$A \otimes \mathbb{Z}_2 $, then there exists the natural isomorphism of $A \otimes \mathbb{Z}_2 $ on $A$ that is defined by the formular
$$
a \otimes \varphi \mapsto a ,\quad a \in A.
$$
\end{proof}


\noindent \textbf{Acknowledgement.} The authors gratefully acknowledge the support from the  RFBR-16-01-00414 and RFBR-15-01-00745. Also, we thank S.~Ivanov,  A.~Lavrenov and V.~Thomas for the interesting discussions and  useful suggestions.

\end{document}